\documentclass{amsart}
\usepackage[foot]{amsaddr}
\usepackage{graphicx}
\def\cal{\mathcal}

\input{tcilatex}

\newcommand{\EE}{{\mathbb E}}

\newcommand{\NN}{{\mathbb N}}
\newcommand{\PP}{{\mathbb P}}
\newcommand{\RR}{{\mathbb R}}
\newcommand{\bP}{P}
\newcommand{\bQ}{{\cal{Q}}}
\newcommand{\mR}{{\cal{R}}}
\newcommand{\bE}{{\bf E}}

\newcommand{\J}{{\cal J}}
\newcommand{\hH}{{\widehat H}}
\newcommand{\aH}{{\cal H}}

\newcommand{\aS}{{\cal S}^{\downarrow}}
\newcommand{\aaS}{{\cal S}}

\newcommand{\X}{{\cal X}}

\newcommand{\mN}{\mathfrak{N}}
\newcommand{\mK}{\cal{P}}
\newcommand{\mL}{\cal{P}}

\newcommand{\opi}{\underline{\pi}}
\newcommand{\oPi}{\underline{\Pi}}
\newcommand{\uoPi}{{\Pi}^*}
\newcommand{\hpi}{\widehat{\pi}}

\input tcilatex
\DeclareUnicodeCharacter{207A}{\ensuremath{^+}}  

\title[Notes]{A note on the convergence of the Bayesian entropy estimator for 
exchangeable partitions}

\author{Servet Martinez$^{1}$}
\thanks{{\it Corresponding author:} Servet Martinez, smartine@dim.uchile.cl}
\address{$^{1}$ Departamento de Ingenier\'{i}a Matem\'{a}tica \\
Centro Modelamiento Matem\'{a}tico\\
UMI 2807, UCHILE-CNRS \\
Casilla 170-3 Correo 3, Santiago, CHILE\\
E-mail: smartine@dim.uchile.cl}

\date{\today}

\begin{document}
\newtheorem{property}{Property}

\begin{abstract}
We show that when the proportions of a countable set of species are organized as an
exchangeable partition of the unit interval and we take a sample on it, then 
the Bayesian posterior entropy 
converges a.s. and in $L^1$ to the entropy of the species when 
the sample size diverges to infinity.  
\end{abstract}

\maketitle

\textbf{Running title}: Posterior entropy estimator for exchangeable partitions.\newline

\bigskip

\noindent {\bf AMS Classification Number:} 62B10, 94A17

\bigskip 

\noindent {\bf Keywords:} Entropy, Bayesian posterior distribution, 
exchangeable partitions, Poisson-Dirichlet Process.

\noindent{\bf Founding Institution:} ANID Basal PIA program FB210005 Center for Mathematical Modeling. 

\bigskip-

\section{Introduction}
\label{sec0}
The study of the diversity of abundance of species  when the number of species and 
their abundance are unknown has been approached by using several models, many of 
which take the Poisson Dirichlet process (PDP) 
as the prior distribution of the abundance of species.  For instance this was done 
in \cite{buntine2012}, and in \cite{favaro} and \cite{sharif2008},  
in ecology and machine learning respectively.
For this prior, the study of the diversity of abundance of species  was studied in 
\cite{chao2003}  by using the Shannon entropy as an index of diversity. 
The estimation of the entropy when data on the observation of 
species is collected was studied in \cite{archer2014}, by considering  
the posterior Bayesian entropy  and the plug-in estimator of the entropy.
Here we study these estimators in the more general framework of exchangeable partitions 
and we sharpen the convergence description of the posterior Bayesian entropy by
using a martingale characterization. 

\medskip

To be more precise let $H(\cdot)$ be the Shannon entropy  of the 
random masses of an exchangeable partition of the unit interval. Take 
a sample of size $n$ by using i.i.d. random variables in $[0,1]$ and denote  
the number of  classes containing the first $n$ observations by $k_n$ and 
the total number of individuals in the $i-$th class by $\pi^n(i)$.
Let $\hH_n$ be the plug-in entropy estimator 
$\hH_n(\cdot)=-\sum_{i=1}^{k^n}(\pi^n(i)/n)\log (\pi^n(i)/n)$ and let
$\aH_n(\cdot)=-\sum_{i\in \NN} q^n_i \log q^n_i$
be the posterior Bayesian entropy which is the entropy of the classes 
distributed as $q^n=(q^n(j): j\in \NN)\sim 
P(\cdot | \pi^n(1),..,\pi^n(k^n))$, 
the posterior distribution of the exchangeable partition given the sample.

\medskip

In \cite{archer2012}, the entropy and its estimators were studied when the prior is a PDP
distribution. The mean of the entropy was given in formula (12),
and an explicit formula for the posterior Bayesian entropy $\aH_n$ was found
in terms of the digamma function, based upon the characterization of the posterior distribution 
$P(\cdot | \pi^n(1),..,\pi^n(k^n))$ given in Corollary $20$ in \cite{pitman1996}. 
Moreover, in Theorem $3$ in  \cite{archer2014},  
it was shown that $|\hH_n-\aH_n|\to 0$ in probability as $n\to \infty$. 
This uses the behavior of the digamma function close to the log for big values
and that $k^n/n\to 0$ in probability (the latter following from Section $3$ and 
Proposition $2$ in \cite{gnedin}). In formula $(13)$ in \cite{archer2014},
is was also shown that the entropy has finite second moment.

\medskip

In this work we enlarge the study of the entropy to the case 
when the prior is an exchangeable random partition
and show the consistency of the entropy estimators.
In Proposition \ref{th1}, we show that when the exchangeable partition has finite mean entropy, 
one has 
\begin{equation}
\label{eq0}
\lim\limits _{n\to \infty}\aH_n(\cdot)=H(\cdot)=\lim\limits_{n\to \infty} \hH_n(\cdot) 
\hbox{ a.s. and in }L^1.
\end{equation}
Moreover, if the second moment of the entropy is finite, the above convergence holds in $L^2$ and
if the entropy has finite $p-$th moment for some $p>1$, the first equality holds in $L^p$.
To prove the first equality in (\ref{eq0}), we show and use that the posterior 
Bayesian entropy is an integrable martingale when the sample size varies in $\NN$. The convergence
of the plug-in estimator follows from Corollary $1$ in \cite{antos}.   
When the second moment is finite, the martingale property gives an increasing process 
with the sample size, see Remark \ref{rem2}.

\medskip

The relation (\ref{eq0}) also holds when one replaces the entropy $H$ by some other function 
$G:\aS\to \RR$ which is symmetric in its arguments and satisfies the moment conditions.
While the result holds for the  Bayesian posterior without any other requirements, for
the plug-in estimator, the result needs $G$ to be  an additive function that satisfies additional conditions.

\medskip

In Section \ref{sec2} we state and show Proposition \ref{th1}, which is our main result.
In Section \ref{sec1} we introduce some of the main notions for exchangeable partitions.
We closely follow the presentation in \cite{bertoin}, but with some minor variation in the way 
the concepts are introduced which is more convenient for presenting and showing the result. 
Finally, we discuss this result in the context of the PDP process in Section \ref{sec3}.

\medskip

\section{Exchangeable partitions}
\label{sec1}
A random partition $\Xi=\{\xi(i)\}$ of $\NN$ is called exchangeable if its law is invariant under
the class of permutations $\{\varphi\}$ of $\NN$ that are the identity after some arbitrary finite 
integer. Then, $\varphi(\Xi)\sim \Xi$ $\; \forall \sigma$, where 
$\varphi(\Xi)=\{\varphi^{-1}(\Xi(i))\}$. In this section we 
give some of the main notions allowing to describe this class of partitions, and we follow 
closely Section 2.3 in \cite{bertoin}. 

\medskip

Let $\aS=\{s=(s(i): i\in \NN): s(i\!+\!1)\ge s(i)\!\ge \!0\;  \forall i\in \NN, \, \sum_{i\in \NN}s(i)=1\}$. 
As done in \cite{bertoin} Section 2.1.2, 
it is useful to associate to each $s\in \aS$ a collection of disjoint open intervals 
$\J^s=(J^s(i): i\in \NN)$ of $[0,1]$ such that the sequence $(|J^s(i)|^{\downarrow} )$ of the 
interval lengths ranked in a decreasing way, is equal to $s$. Let $\bP$ be a probability law on $\aS$. 
The mean expected value with respect to $\bP$ is denoted by $\bE_\bP$. 

\medskip

Let $\X=(X_n: n\ge 1)$ be a sequence of i.i.d. Uniform r.v's in $[0,1]$ independent of $s\sim \bP$. 
The law of $\X$ is denoted by $\mR$, it is invariant under 
the set of permutations $\{\varphi\}$ of $\NN$ that satisfy $\varphi(i)=i$ for some $i>n(\varphi)$.   
So, when writing $\varphi(\X)=(X_{\varphi(n)}: n\in \NN)$, one has
$\varphi(X)\sim \mR$. We will consider the probability law $\PP=\bP\otimes \mR$ of the pair $(s,\X)$,
and the conditional distribution given $s$ is denoted by $\mK(\cdot | s)$. 

\medskip

Define a random sequence $\X^*=\X^*(s,\X)=(X^*_n: n\ge 1)$ taking values in $\NN$ and 
obeying the recursion: 
$X^*_1=1$ and if, up to $n$, the sequence $X_1,...,X_n$ have visited $k^n$ 
different intervals in $\J^s$, then 
\begin{eqnarray*}
&{}& X^*_{n+1}\!=\!j , k^{n+1}\!=\!k^n \hbox{ if for some } j\!\in \!\{1,..,k^n\}, 
X_{n+1} \hbox{ is in the same interval as }X^*_j;\\
&{}& X^*_{n+1}\!=\!k^n\!\!+\!\!1, k^{n+1}\!=\!k^n\!\!+\!\!1 \hbox{ if }
X_{n+1} \hbox{ is in an interval not being visited before }n.
\end{eqnarray*}
Most of the notions will only depend on the probability distribution of the pair $(s,\X^*)$. 
 
\medskip

Let $\xi^n=(\xi^n(1),..,\xi^n(k^n))$ be the partition of $\{1,..,n\}$  
given by $\xi^n(i)=\{j\in \{1,..,n\}: X^*_j=i\}$. The sequence of partition $(\xi^n: n\in \NN)$ 
is compatible (that is $\xi^{n+1}|_{\{1,..,n\}}=\xi_n \; \forall n$) 
and it defines an exchangeable random partition $\Xi=(\xi(i): i\in \NN)$ of $\NN$.
Let $\pi^n(j)=\# \xi^n(j)$ be the number
of elements of $\xi^n(j)$. Define the vector $\pi^n=(\pi^n(1),..,\pi^n(k^n))$ 
ranked by the order of the indexes that firstly visit the new intervals. One has $k_1=1$, 
$\pi^1=(1)$ and $\pi^{n+1}\in \{\pi^n\!+\delta_{k^n}(j): j\in \{1,...,k^n\}\}\cup \{(\pi^n,1)\}$
where $\delta_k(i)$ is the vector of dimension $k$ with all $0'$s except by a $1$ in position $j$.
Define $\Pi=(\pi^n: n\in \NN)$, we have $\pi=\Pi(s,\X)$. A recurrence argument 
shows that the passage from $\pi^n$ to $\pi^{n+1}$ 
determines the value of $X^*_n$, then the random elements $\Pi$, $\X^*$ and
$\Xi$, mutually determine each other. 
 
\medskip

There exists the asymptotic frequencies 
$\hpi(i)=\lim\limits_{n\to \infty} \#(\xi^n(i)\cap \{1,..,n\})/n$ $\,\mK( \cdot | s)$ a.s. 
and $\hpi=(\hpi(i): i\in \NN)$ is a size-biased reordering of $s$, 
see Proposition $2.8$ in \cite{bertoin}. In a reciprocal way, in Kingman theory
it is shown that every random exchangeable 
partition $\Xi$ posseses asymptotic frequencies a.s. given by a vector of frequencies $\hpi$, 
see \cite{kingman} and also \cite{pitman1995}. Then, the
decreasing ranked sequence $\hpi^{\downarrow}$ is distributed with
some law $\bP$ in $\aS$. Hence, since $\Xi$ and $\Pi$ determine one another, 
the law of an exchangeable random partition
can be set in the form $\mK(\Pi |s) d\bP(s)$, see Theorem $2.1$ in \cite{bertoin}.

\medskip
 
Let $\X_n=(X_1,..,X_n)$ and $\Pi_n=(\pi^1,..,\pi^n)$. We have $\Pi_n=\Pi_n(s,\X_n)$. 
Let $\PP_n$ be the joint probability of $(s,\Pi_n)$, then
$d\PP_n(s,\Pi_n)=\mK(\Pi_n | s) d\bP(s)$. 
Let $\bQ$ be the marginal distribution of $\Pi$, so 
$d\PP_n(s,\Pi_n)=\mL(ds | \Pi_n)\bQ(\Pi_n)$, where  
$\mL(\cdot | \Pi_n)$ is the conditional distribution given $\Pi_n$. 
We claim that,  
\begin{equation}
\label{eq-1}
\mL(ds, X^*_{n+1} | \Pi_n)=\mL(ds, X^*_{n+1} | \pi^n). 
\end{equation}
Let us show it. The set of values taking by the variable $\Pi_n(s,\X_n)$ is denoted
by ${\mN}_n$. For a value $\opi^n$ taken by $\pi_n(s,\X_n)$, we set
$\mN_n(\opi^n)=\{\Pi_n=(\pi_n,..,\pi^n)\in \mN_n: \pi_n=\opi_n\}$ and so
$\{\pi^n(s,\X_n)=\opi^n\}=\bigcup_{\oPi_n\in \mN_n(\opi^n)} \{\Pi_n(s,\X_n)=\oPi_n\}$. 
Fix some value $\uoPi_n\in \mN_n(\opi^n)$.
For every $\oPi_n\in \mN_n(\opi^n)$ one can find a permutation $\sigma$ of $\{1,..,n\}$ 
that satisfies $\{\Pi_n(s,\X_n)=\oPi_n\}=\{\Pi_n(s,\sigma(\X_n))=\uoPi_n\}$.
The equality 
$$
\PP(ds, \! X^*_{n+1},\! \Pi_n(\!s,\! \X_n)\!\!=\!\oPi_n)\!\!=\!\PP(ds, \! X^*_{n+1}, 
\! \Pi_n(\!s,\! \sigma(\! \X_n))\!\!=\!\uoPi_n)
\!\!=\!\PP(ds, \! X^*_{n+1},\! \Pi_n(\!s,\! \X_n)\!\!=\!\uoPi_n),
$$ 
gives 
$\PP(ds, X^*_{n+1} |\pi_n(s,\X_n)=\opi_n)=\PP(ds, X^*_{n+1} | \Pi_n(s,\X_n)=\uoPi_n)$, 
so (\ref{eq-1}) follows.

\medskip

Let us use (\ref{eq-1}) to show that
the marginal distribution $\bQ$ has the following Markovian structure,
\begin{equation}
\label{eq-2}
\bQ(\Pi_n)=\prod_{k=1}^{n-1} Q(\pi^k; \pi^{k+1}) \hbox{ with } Q(\pi^k; \pi^{k+1})=\mL(\pi^{k+1} | \pi^k).
\end{equation}
Since $\pi^{n+1}=\pi^{n+1}(\pi^n,X^*_{n+1})$, we use (\ref{eq-1}) to get 
$$
\PP(\Pi_{n+1})\!=\!
\PP(\Pi_n)\mL(\Pi^{n+1} | \Pi_n)
\!=\!\PP(\Pi_n)\mL(\pi^{n+1} | \Pi_n)=\PP(\Pi_n)\mL(\pi^{n+1} | \pi^n).
$$
Then, $\PP(\Pi_{n+1} | \Pi_n)=Q(\pi^n; \pi^{n+1})$. Since $\bQ(\Pi_n)=\PP(\Pi_n)\; \forall n$ and 
$\PP(\Pi_1=(1))=1$, we get (\ref{eq-2}).

\medskip
 
From (\ref{eq-1}) we get 
$d\PP_n(s,\Pi_n)=\mL(ds | \Pi_n)\bQ(\Pi_n)=\mL(ds | \pi^n)\bQ(\Pi_n)$.
Then, every measurable nonnegative $g_n:\aS \times {\mN}_n\to \RR_+$ satisfies 
\begin{eqnarray}
\nonumber
\int\limits_{\aS}\! \sum\limits_{\Pi_n\in \mN_n}  \!\!\! g_n(s,\Pi_n) 
\mK(\Pi_n | s) d\bP(s)=\!\int \!   \! g_n d\PP_n\!&=&\!\!\!
\sum\limits_{\Pi_n\in \mN_n}\int\limits_{\aS} \!\! g_n(s,\Pi_n) \mL(ds | \Pi_n) \bQ(\Pi_n)\\
\label{eq2}
&=&\!\!\!
\sum\limits_{\Pi_n\in \mN_n}\int\limits_{\aS} \!\! g_n(s,\Pi_n) \mL(ds | \pi^n) \bQ(\Pi_n).
\end{eqnarray}
The first two equalities are found in Lemma $1$ in \cite{ishwaran2003}.

\medskip

\section{Main Result}
\label{sec2}
The Shannon entropy of a distribution
$\rho=(\rho_i: i\in I)$ on a countable set $I$ is given by 
$H(\rho) = - \sum_{i\in I}\rho_i \log(\rho_i)$. So, the entropy of $s\in \aS$ is
$H(s)=-\sum_{i\in \NN}s(i) \log(s(i))$ and the mean entropy is 
$\bE_\bP( H)=\bE_\bP(-\sum_{i\in \NN}s(i) \log(s(i)))$. We will assume that
$\bE_\bP( H)$ is finite.

\medskip

The structural distribution of the random partition 
is the law of $\hpi(1)$, we call it $F$. 
In  relation ($25$) in \cite{pitman2-1996} it is shown that every Borel function $g:[0,1]\to \RR_+$,  
satisfies $\bE_\bP\left(\sum_{i\in \NN} g(s(i))\right)=\int_0^1 \frac{g(x)}{x} dx$. 
So, $\bE_\bP(H)=-\int_0^1 \log(x)dF(x)$ and 
$\bE_{\bP}(H)\!< \!\infty$ is equivalent to $\int_0^1 {\log}(x)dF(x)\!> \!-\infty$.

\medskip

The Bayes posterior mean and the plug-in estimator of the entropy at step $n$ are, respectively, 
given by $\bE_{\mL(\cdot| \pi^n)}(H)$ and $H(\pi^n/n)$.

\medskip

Let  $\aaS=\{s=(s(i): i\! \in \! \NN): s(i)\!\ge \!0\;  \forall i\!\in \! \NN, \sum_{i\in \NN}s(i)=1\}$.
A function $G:\aaS\to \RR$, $s\to \RR$, is symmetric (in its arguments) if 
$G(s)=G(s')$ where $s'\in \aS$ is the sequence of the components $(s(i))$ ranked in a decreasing way.
So $G:\aS\to \RR$ and its integrability properties refer to the probability 
$\bP$ in $\aS$.

\medskip

We state our main result for both, the mean posterior and the plug-in 
entropy estimators. But the proof for the plug-in estimator follows directly from \cite{antos}. 

\medskip

\begin{proposition}
\label{th1}
Assume $\bE_{\bP}(H)<\infty$. Then, 
\begin{equation}
\label{eqth}
\lim\limits_{n\to \infty}\bE_{\mL(\cdot| \pi^n)}(H)=H
=\lim\limits_{n\to \infty}  H(\pi^n/n) \; \; \PP\!-\!\hbox{a.s. and in } L^1(\PP).
\end{equation}
If $\bE_{\bP}(H^p)<\infty$, then the first limit holds in $L^p(\PP)$ and if
$\bE_{\bP}(H^2)<\infty$ then the second limit holds in $L^2$.

\medskip

Moreover, when $G:\aaS\to \RR$ is symmetric and belongs to $L^1(\bP)$, 
then 
\begin{equation}
\label{eqth2}
\lim\limits_{n\to \infty}\bE_{\mL(\cdot| \pi^n)}(G)=G \; \; \PP\!-\!\hbox{a.s. and in } L^1(\PP),
\end{equation}
and if $\bE_{\bP}(G^p)<\infty$ for some $p>1$, then the limit holds in $L^p(\PP)$.
\end{proposition}

\begin{proof}
Let us prove (\ref{eqth2}). From hypothesis, 
$\infty > \bE_{\bP}(|G|)=\bE_\PP(|G|)$.
Fix some value $\oPi_n=(\opi^j: j=1,..,n)\in \mN_n$ and 
take $g_n(s,\oPi_n)= G(s) {\bf 1}(\Pi_n(s,\X_n)=\oPi_n)$ in  formula (\ref{eq2}). Then,
$$
\int_{\aS} G(s) \mK(\oPi_n | s)\bP(ds)
=\bQ(\oPi_n)\int_{\aS} G(s) \mL(ds | \opi^n)=\bQ(\oPi_n) \bE_{\mL(\cdot| \opi^n)}(G).
$$
Since 
$$
\int_{\aS} G(s) \mK(\oPi_n | s)\bP(ds)=\int_{\{\Pi_n=\oPi_n\}} \! G \, d\PP \; \hbox{ and } \; 
\bQ(\oPi_n)=\PP(\Pi_n(s,\X_n)=\oPi_n),
$$
we get
\begin{equation}
\label{eqmart1}
\bE_{\mL(\cdot| \opi^n)}(G)=\frac{1}{\PP(\Pi_n=\oPi_n)}\int_{\{\Pi_n=\oPi_n\}} G d\PP=
\EE(G|\sigma(\Pi_n))(\oPi).
\end{equation}
Here $\sigma(\Pi_n)$ is the $\sigma-$field generated by $\Pi$.
Then, $\left(\bE_{\mL(\cdot| \pi_n)}(G): n\in \NN\right)$ is an 
integrable $\PP-$martingale with respect to the filtration of $\sigma-$fields $(\sigma(\Pi_n): n\!\in \!\NN)$.
The limit $\sigma-$field is $\sigma(\Pi)$ ($\PP-$completed). 
We have $\lim\limits_{n\to \infty} \bE_{\mL(\cdot| \pi_n)}(G)=G$ $\, \PP-$a.s. 
because the $\sigma(\Pi)-$measurable sequence $\hpi=\lim\limits_{n\to \infty }\pi^n/n$ 
is a size-biased reordering of $s$ $\; \mK(\cdot | s)$ a.s.
and the symmetry of $G$ gives $G(\hpi)=G(s)$ $\; \mK(\cdot | s)$ a.s. Then, the $\PP-$a.s. 
convergence in (\ref{eqth2}) is satisfied.

\medskip

The martingale theorem for integrable martingales 
(for instance Proposition $II\!-\!2\!- \!11$ in \cite{neveu}) gives the 
$L^1(\PP)$ convergence in (\ref{eqth2}) and when  $\bE_{\bP}(G^p)<\infty$ for some $p>1$, the 
$L^p(\PP)$ convergence in  (\ref{eqth2}) follows from Proposition $II\!-\!2\!-\!11$ in \cite{neveu}.
So, the first equality in (\ref{eqth}) is satisfied and under $\bE_{\bP}(H^p)<\infty$ the assertion 
on the $L^p(\PP)$ convergence holds.

\medskip

Since $\bE_\bP(H)<\infty$ one has $H(s)<\infty$ $\, \bP$-a.s. So, from Corollary $1$ 
in \cite{antos} we get that  $\bP$-a.s.,  $H(\pi^n/n)\to H(s)$ $\; \mK(\cdot | s)$-a.s. Then 
 $H(\pi^n/n)\to H(s)$ $\; \PP$-a.s. 
From  Theorem $1$ and Corollary $1$ in \cite{antos} we get that $\bP$-a.s. for all 
$\epsilon\in (0,1)$ there is a constant $K(\epsilon)$ such that 
$\bE_{\mK(\cdot  | s)} ((H(\pi^n/n)-H(s))^2)\le 2K(\epsilon) n^{-\epsilon}$, so the bound does not depends on $s$. 
(For a sharper bound see Remark $(iv)$ in page \cite{antos} p.168).
Then, $\bE_{\PP} ((H(\pi^n/n)-H(s))^2)\le 2K n^{-\epsilon}$ and in the second equality of (\ref{eqth}),
 the $L^1(\PP)$ and the $L^2(\PP)$ convergence follow.
\end{proof}

\medskip

\begin{remark}
Property (\ref{eqth2}) expresses that the posterior Bayesian property holds for
an integrable symmetric function $G:\aaS\to \RR$. A similar generalization
can be stated for the plug-in estimators, but for a particular class of symmetric integrable functions $G$.  
It is required that they have the additive form $G(s)=\sum_{i\in \NN} f(s(i))$ where $f$ satisfies 
the conditions given in Theorem $1$ in \cite{antos} 
(this Theorem is stated for the class of functions that satisfy $G(s)=\sum_{i\in \NN} f_i(s(i))$, 
but together with symmetry this reduces to $f_i=f$). 
Finally, let us mention that other consistent estimators of the entropy for finite entropy partitions 
having an infinite number of classes, can be found in Sections  $4$ and $5$ in \cite{silva}.
\end{remark}

\medskip

\begin{remark}
\label{rem3}
Let $\opi^n$ be a value taken by $\pi^n(s,\X_n)$. Then,  
$C(\opi^n)=\{\opi^n+\delta_{k_n}(j): j\!=\!1,..,k_n\}\cup  \{(\opi^n,1)\}$
is the set of the values that can take $\pi^{n+1}$ when $\pi^n=\opi^n$. Since 
$$
E({\bf 1}_{\opi^{n+1}}|\Pi_n)=\mL(\opi^{n+1} | \pi^n)
=Q(\pi^n; \opi^{n+1}),
$$ 
the martingale property  
$\bE_{\mL(\cdot| \pi^{n})}(G)=E\left(\bE_{\mL(\cdot| \pi^{n+1})}(G)|\Pi_n\right)$ gives
$$
\bE_{\mL(\cdot| \pi^{n})}(G)=\sum_{\opi^{n+1}\in C(\pi^n)} 
\bE_{\mL(\cdot| \opi^{n+1})}(G)Q(\pi^n; \opi^{n+1}).
$$
\end{remark}

\medskip

\begin{remark}
\label{rem2}
Assume $H^2\in L^2(\PP)$. Then the sequence $(A_n: n\in \NN)$ given by 
$$
A_{n+1}-A_n=E\left((\bE_{\mL(\cdot| \pi^{n+1})}(H))^2|\pi^n\right)-\left(\bE_{\mL(\cdot| \pi^{n})}(H)\right)^2,
$$
is increasing in the sample size, where $E(\cdot |  \pi^n)$ is the mean expected value conditioned to $\pi^n$. 
See Section $VII\!-\!2$ in \cite{neveu}.
Other monotone functions on the sample size that are based upon a 
weighted difference of entropies in successive steps, is found in \cite{martinez}.
There, it is studied in detail their relation to the sample sizes when new species appear for 
the first time.
\end{remark}

\medskip

 \section{Poisson Dirichlet Distribution}
\label{sec3}
An important class of exchangeable partitions is given by the two parameter Poisson Dirichlet 
Process introduced in \cite{pitman1997}, with parameters $0\leq \alpha< 1$ and $\theta>-\alpha$. 
It  is denoted $PDP(\alpha,\theta)$. 
Its size-biased distribution $\pi=(\pi(i): i\in \NN)$ can be described by
a sequence of independent random variables 
$\beta_j\sim \, \text{Beta}(1-\alpha, \theta +\alpha j)$. It is
$\pi(1)= \beta_1$ and $\pi(j)= \beta_j \prod_{i=1}^{j-1}(1-\beta_i)$ for $j\geq 2$.

\medskip

The marginal distribution $\bQ$ satisfies formula (\ref{eq-2}).  In 
\cite{pitman1996} formula $(42)$ (also see relation $(33)$), 
and in \cite{bertoin} Theorem 2.3 and  Corollary 2.6,
it is shown that the transition kernel $Q(\Pi_n; \Pi_{n+1})=Q(\pi^n; \pi^{n+1})$ 
is given by the Pitman formula, 
\begin{equation}
\label{pit1}
Q(\pi^n; \pi^n\!+\!\delta_{k_n}(j))\!=\!
\frac{\pi^n(j) \!-\! \alpha}{\theta \!+\! n}, \, j=1,..,k^n \hbox{ and }
Q(\pi^n; (\pi^n,1)) 
\!=\! \frac{\theta \!+\! \alpha k^n}{\theta\!+ \! n}.
\end{equation}
In Corollary $20$ in \cite{pitman1996} 
it was shown that the posterior distribution 
$\mL(\cdot | \pi^n)$ chooses a random partition $\pi'$ with a law
$(p_1,\dots,p_{k^n},(1-\sum_{j=1}^{k^n} p_j)\pi'')$, where
\begin{eqnarray}
\nonumber
(p_1,\dots,p_{k^n},1\!-\!\sum_{j=1}^{k^n} p_j) 
&\!\sim \!& \text{\normalfont Dirichlet}(\pi^n(1)\!-\!\alpha, \dots, 
\pi^n(k^n)\!-\!\alpha,\theta\!+\!\alpha k^n), \\ 
\label{pit2}
\pi''\!=\!(\pi''(1),\pi''(2),\dots) &\! \sim\!& PDP(\alpha,\theta\!+\!\alpha k^n),
\end{eqnarray}
and they are independent. 

\medskip

In \cite{archer2012} formula $(10)$ and  \cite{archer2014} formula $(12)$ 
it is shown that   
$\bE_\PP(H) = \psi(\theta+1) - \psi(1-\alpha)$, 
where $\psi(x)
=\Gamma'(x)/\Gamma(x)$ is the digamma function. 
Based upon (\ref{pit2}), in \cite{archer2012} formula ($15$) it is shown 
that the posterior entropy $\aH_n$ satisfies,
$(\theta+n)\aH_n = (\theta+n)\psi(\theta+n+1) -(A_n+B_n)$  with
$A_n=(\theta +\alpha k^n)\psi(1-\alpha)$ and 
$B_n=\sum_{i=1}^{k^n}   (\pi^n_i-\alpha)\psi(\pi^n_i-\alpha+1)$.

\medskip

Now, when using (\ref{pit1}) and the property, $x\psi(x+1)=x\psi(x)+1$ for $x>0$,
one proves that the mean expected values of $A_{n+1}$ and $B_{n+1}$ conditioned to $\X_n$, satisfy:
\begin{eqnarray*}
E(A_{n+1} | \X_n)
&=&\frac{\theta+n+1}{\theta+n}A_n-(1-\alpha)\psi(1-\alpha)\frac{\theta+\alpha k^n}
{\theta+n} \hbox{ and }\\ E(B_{n+1} | \X_n)
&=& \frac{\theta+n+1}{\theta+n}B_n+1-\frac{\theta+\alpha k^n}{\theta+n}
+((1-\alpha)\psi(1-\alpha)+1)\frac{\theta+\alpha k^n}{\theta+n}.
\end{eqnarray*}
Then $(\theta+n+1)E(\aH_{n+1} | \X_n)=(\theta+n+1)\psi(\theta+n+1)-
 (\theta+n+1)\frac{A_n+B_n}{\theta+n}$, 
and so the martingale property $E(\aH_{n+1} | \X_n)=\aH_n$ is satisfied. But this does
not imply that the martingale converges in $L^1$ (see the paragraph 
after Proposition II-2-11 in \cite{neveu}). The a.s. and the $L^1$ convergence 
to $H$, follow from Proposition \ref{th1}.

\medskip

\noindent {\bf Acknowledgments.} This work was supported by the Center for 
Mathematical Modeling ANID Basal PIA program FB210005. The author thanks Jorge Silva from 
DII University of Chile for calling my attention to reference \cite{antos}.

\bigskip

Conflict of interest: The author declares that has no conflict of interest.

\medskip

Data Availability: Data sharing is not applicable to this article as not data sets were generated 
or analyzed during the current study.

\medskip

Declaration of generative AI and AI-assisted technologies: The author declares that has not used
this type of technogies during the current study.


\medskip

\begin{thebibliography}{99}

\bibitem{antos}
Antos, Andr\'as and Kontoyannis, Ioannis. Convergence properties of functional estimates 
for discrete distributions.  
Random Structures \& Algorithms (2001), {\bf 19} 3-4, pp. 163-193.

\bibitem{archer2012}
Archer, Evan, Park, Il Memming and Pillow, Jonathan.
Bayesian estimation of discrete entropy with mixtures of stick-breaking 
priors. Advances in Neural Information Processing Systems (2012), {\bf 25}, 
pp. 2015--2023.

\bibitem{archer2014}
Archer, Evan, Park, Il Memming and Pillow, Jonathan.
Bayesian entropy estimation for countable discrete distributions.
The Journal of Machine Learning Research (2014), {\bf 15}, No. 1,
pp. 2833--2868.

\bibitem{buntine2012}
Buntine, Wray and Hutter, Marcus.
A Bayesian view of the {Poisson-Dirichlet} process. arXiv 1007.0296 (2012).

\bibitem{bertoin} Bertoin, Jean. Random fragmentation and coagulation processes.
Cambridge studies in advanced mathematics 102. (2006) Cambridge University Press.

\bibitem{chao2003}
Chao, Anne and Shen, Tsung-Jen.
Nonparametric estimation of {Shannon's} index of diversity when there are
unseen species in sample.
Environmental and ecological statistics (2003), {\bf 10}, 
No. 4, pp. 429--443.

\bibitem{favaro} Favaro, S., Lijoi, A., Mena, R.H. and Prunster , I. Bayesian non-parametric 
inference  for species variety with a two-parametr Poisson-Dirichlet process prior. 
Journal of the Royal Statistical Society: Series B (statistical Methodology) (2009) 
{\bf 71}(5), pp. 993-1008.

\bibitem{gnedin} Gnedin, A., Hansen B. and Pitman, J.  Notes on the occupancy problems  
with infinyely many boxes: general assumptions an power laws.  Probability Surveys 
(2007) {\bf 4}, pp. 146-171.

\bibitem{ishwaran2003}
Ishwaran I. and James, L.
Generalized weighted chinese restaurant processes for the species
sampling mixture models.
Statistica Sinica (2003), {\bf 13}, No. 4, pp. 1211--1236.

\bibitem{kingman} 
Pitman, Jim.
The coalescent. Stochastic Process. Appl. (1982) {\bf 13}, pp. 235-248.

\bibitem{martinez}
Mart\'inez, Servet and Santiba\~nez, Javier. 
One step entropy variation in sequentrial sampling of species for the Poisson-Dirichlet process. 
Acta Appl. Math.(2023)  {\bf 184} 6, 16 p.

\bibitem{neveu}
Neveu, J. Martingales \`a temps discret. (1972) Masson \'Editeurs, Paris.

\bibitem{pitman1995} 
Pitman, Jim.
Exchangeable and partially exchangeable random partitions. Probab. Thory Related Fields 
(1995) {\bf 102}, pp. 145-158. 

\bibitem{pitman1996} 
Pitman, Jim.
Some developments of the Blackwell-Masqueen urn scheme.
Statistics, Probability and Game Theory. 
IMS Lecture Notes-Monograph Series (1996) Vol. 30, pp. 245--267.

\bibitem{pitman2-1996} 
Pitman, Jim.
Random discrete distribution invariant under size-biased permutation.
Advances in Applied Probabilty (1996) {\bf 28} 2, pp. 525-539.

\bibitem{pitman1997}
Pitman, Jim and Yor, Marc.  
The two-parameter Poisson-Dirichlet distribution derived from a stable
subordinator.
The Annals of Probability (1997), pp. 855--900.

\bibitem{silva} Silva, J.  
Shannon entropy estimation in $\infty-$alphabets from convergence results:
stydying plug-in estimators.  
Entropy (2018) {\bf 20} 6, 397; https://doi.org/10.3390/e20060397 

\bibitem{sharif2008}
Sharif-Razavian, Narges and Zollmann, Andreas.
An overview of nonparametric Bayesian models an applications to natural 
language processing (2009). Science, pp. 71-93.


\end{thebibliography}
\end{document}